\documentclass{amsart}%
\usepackage{amscd,amsfonts,amsmath,amsthm,amssymb,amsxtra,}
\usepackage{mathabx,mathrsfs,subfigure}
\usepackage[utf8]{inputenc}
\usepackage{newunicodechar}
\usepackage{newclude}%
\usepackage[mathscr]{eucal}
\usepackage{indentfirst}%
\usepackage{accents}%
\usepackage{latexsym}%
\usepackage{cancel}%
\usepackage{bbold}%
\usepackage{xcolor}%
\usepackage{verbatim}%
\usepackage{graphicx}
\usepackage{enumitem}%
\usepackage{setspace}%

\usepackage{rotating}%

\usepackage{pifont}%

\usepackage[english]{babel}

\usepackage{mathtools}%
\usepackage{thmtools}%
\usepackage{mathrsfs}%
\usepackage{dsfont}%

\usepackage{calc} 
\usepackage{latexsym}
\usepackage{pdfsync}
\usepackage{mdwlist}

\usepackage{tikz}%
\usepackage{tikz-cd}%
\tikzset{
  no line/.style={draw=none,
    commutative diagrams/every label/.append style={/tikz/auto=false}},
  from/.style args={#1 to #2}{to path={(#1)--(#2)\tikztonodes}}}%

\tikzset{
    rotate90/.style={anchor=south, rotate=90, inner sep=.5mm}
}%

\usepackage[all]{xy}

\usepackage[colorlinks, bookmarks=true]{hyperref}%

\numberwithin{equation}{subsection}%
\newtheorem{theorem}{Theorem}[section]
\theoremstyle{plain}

\newtheorem{corollary}[theorem]{Corollary}
\newtheorem{corollary-definition}[theorem]{Corollary-Definition}

\newtheorem{definition}[theorem]{Definition}

\newtheorem{lemma}[theorem]{Lemma}

\newtheorem{proposition}[theorem]{Proposition}

\numberwithin{equation}{section}

\theoremstyle{definition}

\newtheorem{remark}[theorem]{Remark}

\newcommand{\blank}{\hspace{0.04cm} \rule{2.4mm}{.4pt} \hspace{0.04cm} }
\newcommand{\blankd}{\hspace{0.04cm} \rule{1.5mm}{.4pt} \hspace{0.04cm} }

\DeclareMathOperator{\Ker}{\mathrm{Ker}}
\DeclareMathOperator{\Cok}{\mathrm{Coker}}

\DeclareMathOperator{\Hom}{\mathrm{Hom}}
\DeclareMathOperator{\Ext}{\mathrm{Ext}}

\DeclareMathOperator{\Imr}{\mathrm{Im}}

\DeclareMathOperator{\Tr}{\mathrm{Tr}}

\DeclareMathOperator{\ot}{\overset{\rightharpoondown}{\otimes}}

\mathchardef\mhyphen="2D

\newcommand{\injt}{\mathfrak{s}}
\newcommand{\injc}{\mathfrak{q}}

\newcommand{\ModL}{\mathrm{Mod}\mhyphen\Lambda}

\newcommand{\LMod}{\Lambda\mhyphen\mathrm{Mod}}

\newcommand{\ab}{\mathbf{Ab}}

\newcommand{\lra}{\longrightarrow}

\newcommand{\fp}{{\mathrm{fp}}(\ModL, \ab)}

\newcounter{hours}
\newcounter{minutes}

\newcounter{EquationCounter}[subsection]%

\begin{document}
\title[Stable Hom functors and the Bass torsion]{The finite presentation of the stable Hom functors, the Bass torsion, and the cotorsion coradical}

\author{Alex Martsinkovsky}
\address{Mathematics Department\\
Northeastern University\\
Boston, MA 02115, USA}
\email{a.martsinkovsky@northeastern.edu}

\date{\today, \setcounter{hours}{\time/60} \setcounter{minutes}{\time-\value
{hours}*60} \thehours\,h\ \theminutes\,min}
\subjclass[2010]{Primary: 16S90; Secondary: 16D90, 16E30, 18A25}

\thanks{Supported in part by the Shota Rustaveli National Science Foundation of Georgia Grant NFR-18-10849}

\begin{abstract}
We provide necessary and/or sufficient conditions for the stable Hom functors to be finitely presented. When the covariant Hom functor modulo projectives is finitely presented, its defect is isomorphic to the Bass torsion of the fixed argument. When the contravariant Hom functor modulo injectives is finitely presented, its defect is isomorphic to the cotorsion  of the fixed argument. We also give a sufficient condition for the sub-stabilization of the tensor product to be finitely presented. A finite presentation of the tensor product leads to a unexpected application. 
\end{abstract}

\maketitle
\tableofcontents

\section{Introduction}
Throughout this paper, $\Lambda$ is an associative ring with identity, and $\LMod$ is the category of left $\Lambda$-modules. Our main objects of study are certain additive functors $\LMod \to \ab$, where $\ab$ denotes the category of abelian groups. In particular, we want to know when the functors Hom modulo projectives or injectives are finitely presented. More broadly, we want to know when a stabilization of an additive functor is finitely presented.

Recall that a functor is said to be \textit{finitely presented} if it is the cokernel of a natural transformation between representable functors.\footnote{The term ``cokernel'' (like most of the algebraic nomenclature in relation to functors) is understood here in the componentwise sense.} The heuristics is based on the observation that representable functors have properties similar to those of finitely generated projective modules. The fundamental role of such functors was demonstrated in M.~Auslander's seminal paper~\cite{Aus66}. In short, such categories are abelian, they have global dimension~2 (or, only over semisimple rings, 0), and the category of modules can be recovered as a suitable Serre quotient of the functor category. Later, it was shown by R.~Gentle~\cite{Gen91} that the category of finitely presented functors is both complete and cocomplete.

The existence of a finite presentation of an additive functor can be informally thought of as a ``coordinatization'' of the functor, which is analogous to the situation with finitely presented modules: any finitely presented module is completely determined by a matrix. In the functor category, the ``matrix'' is just a module homomorphism (by the Yoneda lemma).

Of special significance are the so-called \textit{stable} functors. These are additive functors defined on the stable categories, i.e., the categories obtained from $\LMod$ by ``modding out" the homomorphisms factoring through projectives or, respectively, injectives. Stable functors are identified with the additive functors on $\LMod$ which vanish on projective or, respectively, injective objects. Examples abound: all higher Ext and Tor functors or, more generally, all higher satellites are such. In fact, given any additive \textit{covariant} functor, one may pass to the largest subfunctor vanishing on injectives or the largest quotient functor vanishing on projectives, which provides an unlimited pool of examples. The results of these procedures will be referred to as the sub-stabilization and, respectively, quot-stabilization of the functor.\footnote{The term ``sub-stabilization'' introduced here was originally known under the name of ``injective stabilization'', and the term ``quot-stabilization'' was known as the ``projective stabilization''. We have opted to eliminate the references to the terms ``injective'' and ``projective'' to avoid the misleading allusions to the namesake resolutions. Indeed, for contravariant functors, the choice of the resolutions is reversed, see below.} A similar construction works for \textit{contravariant} functors by interchanging the prefixes ``sub'' and ``quot'' in the previous construction. Thus defined functors ``measure" the deviation of the functor from its zeroth left- or, respectively, right-derived functor. Clearly, all stable functors arise this way. A systematic study of the related phenomena, known as stable module theory, was initiated in~\cite{AusBr69}. For recent developments, see~\cite{MR-1, MR-2, MR-3}.

As we already mentioned, when an additive functor is finitely presented, one can reduce, to some extent, the study of the functor to the study of modules. Given the ubiquity of stable functors, in this paper we investigate the question of finite presentation of stable univariate Hom functors, of which there are four, corresponding to the choices of the variance of the functor and the projectively or injectively stable category. The established notation is to use overline to denote Hom modulo injectives and underline for Hom modulo projectives. Two of such functors,  $(\underline{A, \blank})$ and $(\overline{\blank, A})$, are actually the quot-stabilizations of the corresponding Hom functors; they may or may not be finitely presented. The other two functors, $(\overline{A, \blank})$ and $(\underline{\blank, A})$
are always finitely presented.\footnote{We emphasize that these two functors are not the stabilizations of the corresponding Hom functors.} There are also the two sub-stabilizations of the univariate Hom functors, but they are always zero (since Hom preserves monomorphisms), and we don't mention them anymore.

Here is the plan of the paper. In Section 2, we briefly recall basic facts about finitely presented functors and the category of such functors. 

In Section~3, we examine the four stable Hom functors and provide necessary and/or sufficient conditions for such functors to be finitely presented. It is in this section that the Bass torsion comes to the fore. Parallel to that, we also see the cotorsion coradical~$\injc$, originally introduced in~\cite{MR-2}. 

In Section~4, we deal with the finite presentation of the tensor product and some closely related functors. In particular, we give a sufficient condition guaranteeing a finite presentation of the sub-stabilization of the tensor product. In fact, we show that the sub-stabilization of any finitely presented covariant functor is finitely presented. We also show that over a noetherian ring with a zero socle (on the other side) all nonzero injectives are infinitely generated.

For the module-theoretic terminology used here, see~\cite{AF92}.

\section{The category of finitely presented functors}

In this section we recall the definition and basic properties of finitely presented functors and of the category they form. For more details, see~\cite{Aus66}.

\begin{definition}
 A functor $F$ is said to be finitely generated if there is a module $X$ and a natural transformation $\alpha : (X, \blank) \to F$ which is epic on each component.\footnote{The symbol $(X, \blank)$ denotes $\Hom(X, \blank)$.}
\end{definition}

\begin{proposition}\label{P:set}
 Let $F$ and $G$ be functors. If $F$ is finitely generated, then the class of all natural transformations $g : F \to G$, denoted by $\mathrm{Nat}(F, G)$, is a set. 
\end{proposition}
\begin{proof}
 Let $\alpha : (X, \blank) \to F$ be a componentwise epimorphism. Then, the assignment $g \mapsto g\alpha$ yields an injection $\mathrm{Nat}(F, G) \to\mathrm{Nat}((X,\blank), F)$. By the Yoneda lemma, the latter is isomorphic to $F(X)$, which is a set. Hence $\mathrm{Nat}(F, G)$ is a set, too.
\end{proof}

\begin{definition}
 A functor $F$ is said to be finitely presented if there are modules $X$ and $Y$ and a homomorphism $f : X \to Y$ such that the sequence
\begin{equation}\label{Eq:fp}
 (Y,\blank) \overset{(f, \blankd)}\lra (X, \blank) \lra F \lra 0
\end{equation}
 is componentwise exact.
\end{definition}

Thus every finitely presented functor is finitely generated and, as a consequence of 
Proposition~\ref{P:set}, the totality of all finitely presented functors and natural transformations between them, denoted here by~$\fp$, is a category.

\begin{proposition}\label{P:fp-ker}
$\fp$ has kernels and cokernels. Both are computed componentwise.
\end{proposition}

\begin{proof}
See~\cite[Proposition~2.1]{Aus66}. For a more direct proof, see~\cite[Proposition~3.1]{Aus82}.
\end{proof}

\begin{proposition}
 $\fp$ is abelian.
\end{proposition}

\begin{proof}
 The canonical morphism from the coimage to image is an isomorphism because it is an isomorphism componentwise.
\end{proof}

\begin{proposition}
 $\fp$ has enough projectives.
\end{proposition}

\begin{proof}
The projectives are precisely the representables, which can be seen by applying the Yoneda lemma.
\end{proof}

\begin{remark}
 The projectives in $\fp$ can also be characterized as the left-exact finitely presented functors. Indeed, the representables are of this form. On the other hand, a finitely presented functor clearly preserves arbitrary products. If it is left-exact, it also preserves finite limits. It follows that it preserves arbitrary limits, and by a theorem of Watts, it must be representable.
 \end{remark}                                                                                                                                                                                                                                                                                                                                                                                                                                                                                                                                                                                                                                                                                                                                                                                                                                                                                                                                                                         
 
 An important invariant of a finitely presented functor is its defect, a notion introduced  (without a name) by Auslander in~\cite[pp.~202 and 209]{Aus66}. We recall it now.
 
\begin{definition}
 Let $F$ be a covariant functor with presentation~\eqref{Eq:fp}. The module
 \[
 w(F) := \Ker f
 \]
 is called the defect of $F$.
\end{definition}

The following known result is an easy consequence of this definition.

\begin{lemma}\label{L:defect-zero}
The defect of a finitely presented functor is zero if and only if the functor vanishes on injectives. \qed
\end{lemma}

The notion of defect for a contravariant functor is defined by reversing the arrows~[idem.]:
\begin{definition}
Let~$F$ be a contravariant functor with presentation 
\[
(\blank, Y) \overset{(\blankd, f)}\lra (\blank, X) \lra F \lra 0.
\]
The module $v(F) := \Cok f$ is called the defect of $F$.
\end{definition}

It is immediate from the definition that $v(F) \simeq F(\Lambda)$. The following lemma is now obvious.

\begin{lemma}
 The defect of a finitely presented contravariant functor is zero if and only if the functor vanishes on the regular module. \qed
\end{lemma}

\begin{remark}
 The first phenomenological study of the defect of a finitely presented functor was undertaken by Jeremy Russell 
 in~\cite{Russ16}.
\end{remark}

\section{Finite presentation of the stable Hom functors}

Given a left $\Lambda$-module $A$ we are going look at both covariant and contravariant Hom modulo projectives and injectives.

\subsection{The functor $(\overline{A, \blank})$ is finitely presented}

$(\overline{A, \blank})$ is the quotient of $(A, \blank)$ by the subfunctor $I(A, \blank)$ of all  maps factoring through injectives. The latter, being a subfunctor of the additive functor $(A, \blank)$ is itself additive and therefore $(\overline{A, \blank})$ is also additive. Taking a cosyzygy sequence $0 \to A \to I \to \Sigma A \to 0$ and applying the Yoneda embedding we have an exact sequence of functors 
\[
0 \to (\Sigma A, \blank) \to (I, \blank) \to (A, \blank).
\]
Each component of the last natural transformation has its image in the maps factoring through injectives. On the other hand, any map with domain $A$ factoring through an injective obviously extends to $I$, which shows that the image of the last component is exactly $I(A, \blank)$, the functor of maps factoring through injectives. This yields an exact sequence 
\[
0 \to (\Sigma A, \blank) \to (I , \blank) \to (A, \blank) \to (\overline{A, \blank}) \to 0
\]
and shows that the covariant Hom modulo injectives is finitely presented. 

\begin{remark}
 The functor $(\overline{\Lambda, \blank})$ is precisely the cotorsion coradical $\injc$ introduced in~\cite{MR-2}.
\end{remark}

\subsection{The functor $(\underline{A, \blank})$}

This functor is defined by the exact sequence
\[
0 \lra P(A, \blank) \lra (A, \blank) \lra (\underline{A, \blank}) \lra 0,
\]
where $P(A, \blank)$ applied to a module $B$ produces all maps 
$A\to B$ that factor through projectives. Recall that $(\underline{A, \blank})$ is finitely presented if and only if $P(A, \blank)$ is finitely generated~\cite[p.~130]{Aus82}. We shall see there are rings and modules $A$ for which $P(A, \blank)$ is not finitely generated, but showing this requires some work.

We begin by recalling the notion of the Bass torsion $\mathfrak{t}$. This is the subfunctor of the identity functor on $\LMod$ whose value on~$A$ is defined by the exact sequence $0 \to \mathfrak{t}(A) \to A \to A ^{\ast\ast}$, where $A \to A^{\ast\ast}$ is the canonical evaluation map into the double dual of $A$.\footnote{By extending the codomain of $\mathfrak{t}$ from $\LMod$ to $\ab$ (this amounts to precomposing $\mathfrak{t}$ with the forgetful functor $(\Lambda, \blank)$), we can also view it as a subfunctor of the forgetful functor.} We shall refer to $\mathfrak{t}(A)$ as the Bass torsion submodule of~$A$. It consists of all elements of~$A$ on which every linear from on $A$ vanishes. It also coincides with the reject in $A$ of the regular left module or, equivalently, the reject in $A$ of the class of projectives. Modules with Bass torsion zero are classically referred to as torsionless modules. 

\begin{definition}\label{F:torsion}
 If $\mathfrak{t}(A) = A$, we shall say that $A$ is (Bass) torsion.
\end{definition}

It is not difficult to see that $A$ is torsion if and only if $A^{\ast} = 0$.

\begin{proposition}
 The Bass torsion $\mathfrak{t}$ is a radical, i.e., $\mathfrak{t}$ is a subfunctor of the identity functor on modules, such that $\mathfrak{t}(A/\mathfrak{t}(A)) = 0$ for any module $A$.
\end{proposition}

\begin{proof}
 This is a general property of rejects. 
\end{proof}

The module $A/\mathfrak{t}(A)$ will be referred to as the torsionless quotient module of $A$. Since $\mathfrak{t}$ is a subfunctor of the identity functor $\mathbf{1}$, Stenstr\"{o}m's notation $\mathfrak{t}^{-1}(A) := A/\mathfrak{t}(A)$ makes sense and we have a short exact sequence of endofunctors
\[
0 \lra \mathfrak{t} \lra \mathbf{1} \lra \mathfrak{t}^{-1} \lra 0.
\]

Assuming now that $P(A, \blank)$ is finitely generated, we have an epimorphism 
$(X, \blank) \to P(A, \blank)$ for some module $X$. Precomposing it with the inclusion $P(A, \blank) \to (A, \blank)$ we have a natural transformation $(\gamma, \blank) : (X, \blank) \to (A, \blank)$ induced, by the Yoneda lemma, by a uniquely determined $\gamma : A \to X$. Because the image of 
$(\gamma, \blank)$ is precisely $P(A, \blank)$, this map has the following properties:

\begin{itemize}

\item[a)] For any map $X \to B$, its composition with $\gamma$ factors through a projective (because the image of $(\gamma, \blank)$ is contained in $P(A, \blank)$). 

 \item[b)] Any map $ A \to B$ that factors through a projective extends to $X$ over 
 $\gamma$ (because $P(A, \blank)$ is contained in the image of $(\gamma, \blank)$).
 
\end{itemize}
For brevity, we shall refer to any such map $\gamma : A \to X$ as an approximation of $A$.

Conversely, if there is a module $X$ and an approximation map 
$\gamma : A \to X$ with the above properties, then the induced natural transformation $(\gamma, \blank) : (X, \blank) \to (A, \blank)$ has $P(A, \blank)$ as its image, making $P(A, \blank)$ finitely generated. 

Applying the condition a) to the identity map $ X = X$ we have that $\gamma$ factors through a projective. Conversely, if $\gamma$ factors through a projective, then the  condition a) is satisfied. Thus a) can be replaced by a*): $\gamma$ factors through a projective. Likewise, b) can be replaced by b*): any map from $A$ to a projective extends to~$X$ over $\gamma$. In summary: $P(A, \blank)$ is finitely generated if and only if there is a map $\gamma : A \to X$ satisfying  a) and b) or, equivalently, a*) and b*). As an immediate consequence of this observation, we have 
\begin{proposition}\label{P:qF}
If $\Lambda$ is quasi-Frobenius, then $P(A, \blank)$ is finitely generated for all $A$.
\end{proposition}

\begin{proof}
Just embed~$A$ in an injective $X$, which gives the desired map $\gamma : A \to X$, and then use the defining property of injectives to extend the required maps.
\end{proof}

In the language of commutative diagrams, b*) is the following problem (read from left to right):

\[
\begin{tikzcd}
	A \ar[r, "\gamma"] \ar[d, "\forall \alpha"']
	& X \ar[dl, dashed, "\exists \beta"]
\\
	\forall P
\end{tikzcd}
\]
where $P$ is projective. Assuming a*), if $\delta : Q \to X$ is a projective precover, then~$\gamma$ must factor through $Q$. Assuming b*), it follows that $\alpha$ factors through $Q$. This shows that we may replace~$X$ with the projective $Q$.  Our problem can now be stated as follows. For each module $A$ find a projective~$Q$ and a map $\gamma : A \to Q$ such that, for any projective $P$, any map $\alpha : A \to P$ extends over $\gamma$:
\[
\begin{tikzcd}
	A \ar[r, "\gamma"] \ar[d, "\forall \alpha"']
	& Q \ar[dl, dashed, "\exists \beta"]
\\
	\forall P
\end{tikzcd}
\]

Assuming that a map $\gamma$ with the desired properties exists, we are going to determine its kernel. Since the codomain of $\gamma$ is the projective $Q$, $\gamma$ vanishes on the torsion submodule $\mathfrak{t}(A)$, i.e., $\mathfrak{t}(A) \subseteq \Ker \gamma$. On the other hand, the commutative diagram above shows that $\Ker \gamma$ is contained in $\Ker \alpha$ for all $\alpha$. Since $\mathfrak{t}(A) = \cap_{P, \alpha} \Ker \alpha$, we have $\Ker \gamma \subseteq \mathfrak{t}(A)$, and therefore 
$\Ker \gamma =  \mathfrak{t}(A)$. As a result, we have an exact sequence $0 \to \mathfrak{t}(A) \to A \overset{\gamma}\to Q$ 
with~$Q$ projective and therefore a commutative diagram 
\[
\begin{tikzcd}
 	0 \ar[r ]
	& \mathfrak{t}(A) \ar[r] 
	& A \ar[r] \ar[dr, "\gamma"']
	& \mathfrak{t}^{-1}(A) \ar[r] \ar[d, "\gamma'"]
	& 0
\\
	&
	&
	& Q
\end{tikzcd}
\]
 where the row is exact and $\gamma'$ is monic. The next result is now obvious.
 
\begin{proposition}\label{P:reduction}
 $\gamma$ is an approximation for $A$ if and only if  $\gamma'$ is an approximation for  $\mathfrak{t}^{-1}(A)$. \qed
\end{proposition}

We remark, that the triangle in the above diagram is an epi-mono factorization of $\gamma$.
 
\begin{corollary}
 $P(A, \blank)$ is finitely generated for all modules $A$ if and only if it is finitely generated for all torsionless modules~$A$. \qed
\end{corollary}
 
\begin{corollary}\label{C:embed}
 Let $A$ be a torsionless module. If $P(A, \blank)$ is finitely generated, then $A$ embeds in a projective, i.e., $A$ is a syzygy module. \qed
\end{corollary}

Assume now that $A$ is torsionless and $P(A,\blank)$ is finitely generated so that we have an approximation $\gamma : A \to Q$ with~$Q$  projective and the corresponding short exact sequence $0 \to A \overset{\gamma}{\to} Q \to M \to 0$. Then, for any projective $P$, the induced map $(\gamma, P) : (Q, P) \to (A,P)$ is epic. Applying the functor $(\blank, P)$ to the approximation and passing to the long exact sequence, we have that $\Ext^{1}(M, \blank)$ vanishes on projectives. 

Reversing the foregoing argument, we see that a first syzygy module $A$ of a module $M$ such that $\Ext^{1}(M, \blank)$ vanishes on projectives has the property that $P(A, \blank)$ is finitely generated. This leads to the following criterion.

\begin{theorem}
The functor $(\underline{A, \blank})$ is finitely presented if and only if the torsionless quotient module $\mathfrak{t}^{-1}(A) = A/\mathfrak{t}(A)$ of~$A$ is a first syzygy module of a module $M$ such that $\Ext^{1}(M, \blank)$ vanishes on projectives. \qed
\end{theorem}

Rephrasing the just proved result, $(\underline{A, \blank})$ is finitely presented if and only if there is  an exact sequence
\begin{equation}
 0 \to \mathfrak{t}(A) \to A \to Q \to M \to 0,
\end{equation}
where $Q$ is projective, the first map is the canonical inclusion, and $\Ext^{1}(M, \blank)$ vanishes on projectives. In that case, the epi-mono factorization of the approximation is given by 
$A \to  \mathfrak{t}^{-1}(A) \to Q$. The exact sequence above gives rise to an exact sequence of functors
\[
0 \to (M, \blank) \to (Q, \blank) \to (A, \blank) \to (\underline{A, \blank}) \to 0
\]
and also shows that, when $(\underline{A, \blank})$ is finitely presented, the defect $w(\underline{A, \blank})$ of $(\underline{A, \blank})$ is isomorphic to $\mathfrak{t}(A)$.

\begin{remark}
 In~\cite{Mar24}, the notion of defect was extended to arbitrary additive functors, and it was shown that $w(\underline{A, \blank}) \simeq \mathfrak{t}(A)$ without any assumption on $(\underline{A, \blank})$ or on $A$
\end{remark}
\medskip

We now specialize to the case when $\Lambda$ is left hereditary. Then the image of $\gamma : A \to Q$ must be projective and we may, by Proposition~\ref{P:reduction}, replace $Q$ with that image and set $M := 0$. This proves

\begin{theorem}\label{T:her-cov-hom}
 Suppose $\Lambda$ is left hereditary. Then $(\underline{A, \blank})$ is finitely presented if and only if $\mathfrak{t}^{-1}(A) = A/\mathfrak{t}(A)$ is projective. In that case, $A \simeq \mathfrak{t}(A)  \coprod \mathfrak{t}^{-1}(A)$, $\mathfrak{t}(A)$ is torsion (i.e., the inclusion 
 $\mathfrak{t}^{2}(A) \to \mathfrak{t}(A)$ is an isomorphism)\footnote{The reader is warned that, in general, the torsion submodule need not be torsion, see Definition~\ref{F:torsion}. In other words, the Bass torsion radical is not idempotent.}, and  $(\underline{A, \blank}) = (\mathfrak{t}(A), \blank)$. \qed
\end{theorem}

We can now characterize left hereditary rings over which all functors 
$(\underline{A, \blank})$ are finitely presented. 

\begin{proposition}
 Assume that $\Lambda$ is left hereditary. Then $(\underline{A, \blank})$ is finitely presented for all left modules $A$ if and only if $\Lambda$ is left perfect and right coherent. In particular, this holds if $\Lambda$ is the path algebra of a finite quiver without oriented cycles over a field.
 \end{proposition}
  
\begin{proof}
Assume that  $P(A, \blank)$ is finitely generated for all left modules $A$. Since any product of projectives is torsionless, it embeds, by Corollary~\ref{C:embed}, in a projective and is therefore projective. It is a well-known theorem of Chase~\cite[Theorem~3.3]{Chase} that this forces $\Lambda$ to be left perfect and right coherent. 

Conversely, suppose $\Lambda$ is left perfect and right coherent and let $A$ be an arbitrary module. Since $\mathfrak{t}$ is a radical, $\mathfrak{t}^{-1}(A)$ is torsionless and hence embeds in a product of copies of the regular module which, by the same theorem of Chase, is projective. 
Since~$\Lambda$ is left hereditary, $\mathfrak{t}^{-1}(A))$ must be projective and we have a direct sum decomposition 
$A \simeq \mathfrak{t}(A) \coprod \mathfrak{t}^{-1}(A)$. 
The first claim now follows from Theorem~\ref{T:her-cov-hom}. The second claim follows from the facts that the path algebra in question is hereditary and finite-dimensional over the field.
\end{proof}

\begin{remark}
 The last assertion of the just proved proposition shows that the converse to Proposition~\ref{P:qF} is not true.
\end{remark}

\subsection{The functor $(\underline{\blank, A})$ is finitely presented}

$(\underline{\blank, A})$ is the quotient of the contravariant functor $(\blank, A)$ by the subfunctor $P(\blank, A)$ of all  maps factoring through projectives. The latter, being a subfunctor of the additive functor $(\blank, A)$ is itself additive and therefore 
$(\underline{\blank, A})$ is also additive. Taking a syzygy sequence 
$0 \to \Omega A \to P \to A \to 0$ and applying the Yoneda embedding we have an exact sequence of functors 
\[
0 \to (\blank, \Omega A) \to (\blank, P) \to (\blank, A).
\]
Each component of the last natural transformation has its image in the maps factoring through prjectives. On the other hand, any map with codomain $A$ factoring through a projective obviously lifts to $P$, which shows that the image of the last component is exactly $P(\blank, A)$, the functor of maps factoring through projectives. This yields an exact sequence 
\[
0 \to (\blank, \Omega A) \to (\blank, P) \to (\blank, A) \to (\underline{\blank, A}) \to 0
\]
and shows that the contravariant Hom modulo projectives is finitely presented. 

\subsection{The functor $(\overline{\blank, A})$}

The functor contravariant $\Hom$ modulo injectives can be treated in a manner ``dual" to that of the covariant $\Hom$ modulo projectives. To this end, projectives should be replaced with injectives and arrows need to be reversed. To come up with a notion dual to the quotient modulo Bass torsion, we remark that the Bass torsion of a module is precisely the reject of projectives in that module. For the desired dual, which should be a subfunctor of the identity functor, we  take the trace $Tr(\mathscr{I}, \blank)$ of injectives, i.e., the subfunctor whose value on a module is the trace of the class of all injectives in that module. Borrowing the ``reciprocal'' notation for quotients of the identity functor in~\cite{MR-2}, we denote it by $\injc^{-1}(A)$. This reflects the fact that the cokernel of the inclusion $\injc^{-1}(A) \to A$, called the cotorsion quotient module of $A$, was denoted by $\injc(A)$ in~\cite{MR-2}. Accordingly, we shall refer to $\injc^{-1}(A)$ as the cotorsion-free submodule of $A$. 

\begin{remark}
 The cotorsion functor $\injc$ is a coradical on $\LMod$, i.e., a radical on the opposite category. This is a general property of quotients of the identity endofunctor by traces.
\end{remark}

Let $A$ be a left $\Lambda$-module. To claim that $(\overline{\blank, A})$ is finitely presented is equivalent to saying that the functor $I(\blank, A)$ of maps factoring through injectives is finitely generated. Similar to the arguments before, this amounts to saying that there is a module $Y$ and a map $\gamma : Y \to A$ such that $\gamma$ factors through an injective and any map with codomain $A$ factoring through an injective lifts against $\gamma$. As an immediate consequence of this observation, we have 
\begin{proposition}\label{P:qF-bis}
If $\Lambda$ is quasi-Frobenius, then $I(\blank, A)$ is finitely generated for all $A$. 
\end{proposition}

\begin{proof}
Just cover~$A$ with a projective~$Y$, which gives the desired map $\gamma : Y \to A$, and then use the defining property of projectives to lift the required maps.
\end{proof}

Repeating, with necessary modifications, previously used arguments we have an approximation $\gamma : I \to A$ with an injective $I$. The image of $\gamma$ is in $Tr(\mathscr{I}, A) = \injc^{-1}(A)$ and we have a commutative diagram 
\[
\begin{tikzcd}
 	& \injc^{-1}(A) \ar[d, rightarrowtail]
\\
	I \ar[ur, "\gamma'"] \ar[r, "\gamma"']
	& A	
\end{tikzcd}
\]
Since any map from an injective to $A$ lifts against $\gamma$,  the image of $\gamma$ contains the image of the map and, therefore, it  contains the injective trace $\injc^{-1}(A)$. Since the image of~$\gamma$ is clearly contained in the injective trace, the image of $\gamma$ is precisely the injective trace of $A$ and the diagram above is an epi-mono factorization of $\gamma$. Similar to the previous arguments, one shows that~$\gamma$ is an approximation for~$A$ if and only if $\gamma'$ is an approximation for $\injc^{-1}(A)$. Thus we can concentrate on approximations $\gamma : I \to A$, where $A$ coincides with its own injective trace\footnote{Thus $\injc(A) = 0$.}, and assume that $\gamma$ is epic. As a result, $A$ admits an approximation in the form of a short exact sequence 
$0 \to N \to I \overset{\gamma}{\to} A \to 0$. Mapping an arbitrary injective into this sequence, we have that $\Ext^{1}(\blank, N)$ vanishes on injectives. Reversing the foregoing argument, we see that a first cosyzygy module $A$ of a module $N$ such that $\Ext^{1}(\blank, N)$ vanishes on injectives has the property that $I(\blank, A)$ is finitely generated.

Recall that modules which coincide with their own injective trace are said to be cotorsion-free~\cite[Definition~3.20]{MR-2}. We thus have the following criterion.

\begin{theorem}\label{T:contra-fp}
The functor $(\overline{\blank, A})$ is finitely presented if and only if the cotorsion-free submodule $Tr(\mathscr{I}, A) = \injc^{-1}(A)$ of $A$ is a first cosyzygy module of a module $N$ such that $\Ext^{1}(\blank, N)$ vanishes on injectives. \qed
\end{theorem}

Rephrasing the just proved result, $(\overline{\blank, A})$ is finitely presented if and only if there is an exact sequence
\begin{equation}
 0 \to N \to I \to A \to \injc(A) \to 0,
\end{equation}
where $I$ is injective, the last map is the canonical surjection, and $\Ext^{1}(\blank, N)$ vanishes on injectives. In that case, the epi-mono factorization of the approximation is given by 
$I \to  \injc^{-1}(A) \to A$.
The exact sequence above gives rise to an exact sequence of functors
\[
0 \to (\blank, N) \to (\blank, I) \to (\blank, A) \to (\overline{\blank, A}) \to 0
\]
and also shows that the defect $v(\overline{\blank, A})$ of $(\overline{\blank, A})$ is isomorphic to $\injc(A)$.
\medskip

We now specialize the case when $\Lambda$ is left hereditary. Then the image of $I \to A$ must be injective and we can replace $I$ with that image and set $N := 0$. This proves

\begin{theorem}\label{T:her-contra-hom}
 Suppose $\Lambda$ is left hereditary. Then $(\overline{\blank, A})$ is finitely presented if and only if $\injc^{-1}(A)$ is injective. In that case, 
$A \simeq \injc(A) \coprod \injc^{-1}(A)$, $\injc(A)$ is cotorsion (i.e., the canonical epimorphism $\injc(A) \to \injc^{2}(A)$ is an isomorphism), and $(\overline{\blank, A}) = (\blank, \injc(A))$. \qed
\end{theorem}

We can now characterize left hereditary rings over which all functors 
$(\overline{\blank, A})$ are finitely presented. 

\begin{proposition}
 Assume that $\Lambda$ is left hereditary. Then $(\overline{\blank, A})$ is finitely presented for all left modules $A$ if and only if $\Lambda$ is left noetherian. 
\end{proposition}

\begin{proof}
 Assume that $I(\blank, A)$ is finitely generated for all left modules $A$. Since any direct sum of injectives is cotorsion-free (i.e., it coincides with its own injective trace), it is a first cosyzygy module in an injective resolution by Theorem~\ref{T:contra-fp}, and is therefore injective. It is well-known that this implies that $\Lambda$ is left noetherian. 
 
 Conversely, suppose $\Lambda$ is left noetherian and let $A$ be an arbitrary module. As $\injc^{-1}(A)$ is cotorsion-free, it is a homomorphic image of a direct sum of injectives. The latter is injective by the noetherian assumption. Since $\Lambda$ is assumed hereditary, $\injc^{-1}(A)$ is injective and we have a direct sum decomposition $A \simeq \injc(A) \coprod \injc^{-1}(A)$. In general, it need not be true that the cotorsion quotient module of a module is cotorsion, but it is true for $\injc(A)$ because of the direct sum decomposition -- no injective can be mapped nontrivially into $\injc(A)$ because otherwise the cotorsion-free submodule of $A$ would be strictly bigger than 
$\injc^{-1}(A)$. The desired result now follows from  Theorem~\ref{T:her-contra-hom}.
\end{proof}

\begin{remark}
The just proved proposition shows that the converse to Proposition~\ref{P:qF-bis} is not true.
\end{remark}

\section{The sub-stabilization of the tensor product}\label{S:sub-ot}

In this section we shall give a sufficient condition for the sub-stabilization of the tensor product to be finitely presented.\footnote{The quot-stabilization of the tensor product is zero since the tensor product preserves epimorphisms.}

First, we provide a short proof of the known result that the defect of a finitely presented functor $F$ represents the zeroth right-derived functor of $F$.

\begin{lemma}
If $F$ is a finitely presented functor, then $R^{0}F \simeq \big(w(F), \blank\big)$.
\end{lemma}
\begin{proof}
 Start with a finite presentation
\[
 0 \lra (Z, \blank) \lra (Y,\blank) \overset{(f, \blankd)}\lra (X, \blank) \lra F \lra 0
\]
of $F$ and pass to the corresponding exact sequence 
\[
0 \lra w(F) \lra X \overset{f}\lra Y \lra Z \lra 0.
\] 
Applying the contravariant Yoneda embedding and using the universal property of cokernels, we have a canonical natural transformation  
$\rho_{F} : F \lra \big(w(F), \blank\big)$, which is obviously an isomorphism on injectives. Since its codomain is left-exact, $\rho_{F}$ is the zeroth right-derived transformation and $R^{0}F \simeq \big(w(F), \blank\big)$.
\end{proof}

We now have

\begin{proposition}
 The sub-stabilization of a finitely presented functor is finitely presented.
\end{proposition}

\begin{proof}
 The sub-stabilization of an additive functor is defined by the exact sequence
 \[
 0 \lra \overline{F} \lra F \lra R^{0}F.
 \]
 If $F$ is finitely presented, then $R^{0}F$ is representable, hence finitely presented, and we are done by Proposition~\ref{P:fp-ker}.
\end{proof}

\begin{corollary}
 If $A$ is a finitely presented right $\Lambda$-module, then the sub-stabilization $A \ot \blank$ of $A \otimes \blank$ is finitely presented.
\end{corollary}

\begin{proof}
 Indeed, the duality for finitely generated projective modules implies that if~$A$ is finitely presented, then $A \otimes \blank$ is finitely presented, too.\footnote{The converse of this statement is also true: if $A \otimes \blank$ is finitely presented, then so is $A$.}
\end{proof}

\begin{remark}
 The converse of the corollary is not true. To see that, pick an infinite projective $P$. Then $P \otimes \blank$ is an exact functor, and therefore the zeroth right-derived transformation is an isomorphism, forcing the sub-stabilization $P \ot \blank$ to be zero, which is, of course, finitely presented.
\end{remark}

For completeness, we give a finite presentation of $A \ot \blank$ in the case $A$ is finitely presented. For a general finitely presented $F$, see~\cite[p.~453]{MR-1}.

Let $P_{1} \overset{f}\lra P_{0} \lra A \lra 0$ be a finite presentation of $A$ with the $P_{i}$ finitely generated projectives. Applying the tensor ptoduct, we have an exact sequence of functors
\[
P_{1} \otimes \blank \lra P_{0} \otimes \blank \lra A \otimes \blank \lra 0.
\]
Using the duality for finitely generated projectives and the left-exactness of the Hom, we have a projective resolution of $A \otimes \blank$.
\[
0 \lra  (\Tr A, \blank) \lra (P_{1}^{\ast}, \blank) \overset{(f^{\ast}, \blankd)}\lra (P_{0}^{\ast}, \blank) \lra A \otimes \blank \lra 0.
\]
This presentation comes from the exact sequence
\[
0 \lra A^{\ast} \lra P_{0}^{\ast} \overset{f^{\ast}}\lra P_{1}^{\ast} \lra \Tr A \lra 0.
\]
As an immediate consequence, we have 
\begin{lemma}\label{L:A*}
 If $A$ is finitely presented, then $w(A \otimes \blank) \simeq A^{\ast}$.\footnote{Using the extension of defect to arbitrary additive functors~\cite{Mar24}, one can remove the assumption that $A$ be finitely presented.} \qed
\end{lemma}
Furthermore, we have
$\Imr f^{\ast} \simeq \Omega \Tr A$ and, according to~\cite[p.~453]{MR-1}, we have the desired finite presentation of $A \ot \blank$:
\[
0 \lra  (\Tr A, \blank) \lra (P_{1}^{\ast}, \blank) \lra (\Omega \Tr A, \blank) \lra A \ot \blank \lra 0.
\]
We also notice that the long cohomology exact sequence associated with the short exact sequence 
\[
0 \lra \Omega \Tr A \lra P_{1}^{\ast} \lra \Tr A \lra 0
\]
recovers the well-known isomorphism 
\[
A \ot \blank \simeq \Ext^{1}_{\Lambda^{op}}(\Tr A, \blank).
\]
for a finitely presented $A$.
\smallskip

Lemma~\ref{L:A*} has a surprising application. By Lemma~\ref{L:defect-zero}, the defect of a finitely presented functor is zero if and only if it vanishes on injectives. In particular, if $A$ is finitely presented and torsion (i.e., $A^{\ast} = 0$), then $A\otimes \blank$ vanishes on injectives. This observation leads to

\begin{theorem}
Suppose the Jacobson radical $J$ of $\Lambda$ is finitely generated as a right $\Lambda$-module and $\mathrm{Soc}\,\Lambda_{\Lambda}$, i.e., the left annihilator of $J$, is trivial. Then every nonzero injective left $\Lambda$-module is infinitely generated.
\end{theorem}

\begin{proof}
 Set $A := \Lambda_{\Lambda}/J$. The assumptions on $A$ show that $A$ is finitely presented and $A^{\ast} = 0$. Thus 
 $\Lambda_{\Lambda}/J \otimes \blank$ vanishes on injectives. The desired result now follows from Nakayama's lemma. 
\end{proof}

\begin{corollary}
 Suppose $R$ is a noetherian commutative local ring of depth at least one. Then every nonzero injective $R$-module is infinitely generated. \qed
\end{corollary}
\smallskip

Closely related to the injective stabilization of the tensor product is the torsion radical $\injt$ introduced in~\cite{MR-2}. For a right 
$\Lambda$-module $A$ its torsion $\injt(A)$ is defined as $A \ot \Lambda$, i.e., the value of the sub-stabilization $A \ot \blank$ at the regular left $\Lambda$-module. In short, $\injt = \blank \ot \Lambda$.
An alternative description of $\injt$ is provided by the following construction. Embed $\prescript{}{\Lambda}\Lambda$ in an injective $I$, say, $\iota : \Lambda \lra I$, apply $A \otimes \blank$, and take the kernel. Thus $\injt$ is defined by the exact sequence
\[
0 \lra \injt \lra  \blank \otimes \Lambda \overset{\blank \otimes \iota}\lra \blank \otimes I.
\]
Since tensoring with a projective is an exact functor, $\injt$ is projectively stable (i.e., vanishes on projectives). The defining sequence now yields
\begin{proposition}
 If the injective envelope of $\prescript{}{\Lambda}\Lambda$ is finitely presented, then $\injt$ is finitely presented. In particular, this is true when $\Lambda$ is an artin algebra. \qed
\end{proposition}

\end{document}